\documentclass{amsart}
\usepackage[T2A]{fontenc}
\usepackage[cp1251]{inputenc}
\usepackage[english]{babel}
\usepackage{amssymb, amsmath}
\renewcommand{\le}{\leqslant}
\renewcommand{\ge}{\geqslant}

\newtheorem{thL}{\indent Theorem}

\newtheorem{propos}{\indent Proposition}
\newtheorem{theorem}{\indent Theorem}
\newtheorem*{lemma}{\indent Lemma}
\theoremstyle{definition}

\newtheorem{remark}{\indent Remark}

\title[H\"{o}lder and Minkowski Type Inequalities with Alternating Signs]{H\"{o}lder and Minkowski Type Inequalities\\with Alternating Signs}
\author{Petr Chunaev}
\date{\today}

\begin{document}
\maketitle
\begin{abstract}
We obtain new inequalities with alternating signs of H\"{o}lder and Minkowski type.
\end{abstract}

\section{Introduction}

Most classical inequalities are essentially concerned with positive terms. On the other hand, in different branches of analysis there is necessity to deal with sums and series with alternating signs. The main goal of this work is to obtain inequalities of H\"{o}lder and Minkowski type for such sums and series.

Let us start with formulating several known results. The first inequality with alternating signs was due to G. Szeg\H{o}.

\begin{thL}[Szeg\H{o} (1950) \cite{Szego}]
Let $0\le b_{2n+1}\le b_{2n}\le \cdots \le b_2 \le b_1$ and a~function $f=f(x)$ be convex on $[0;b_1]$, then
\begin{equation}
\label{Sz1}
f\left(\sum_{k=1}^{2n+1}(-1)^{k+1}b_k\right)\le \sum_{k=1}^{2n+1}(-1)^{k+1}f(b_k).
\end{equation}
\label{tA}
\end{thL}
Later, R.~Bellman using a simple geometrical method proved (\ref{Sz1}), where $2n+1$ was replaced by $2n$ and convex $f$ was such that $f(0)\le 0$ (see \cite{Bellman1953} and \cite{Wright}). However, this result has been already contained in Szeg\H{o}'s theorem. Indeed, it is sufficient to put $a_{2n+1}=0$ in (\ref{Sz1}) and take into account that $f(0)$ in the right hand side is non-positive.
In \cite{Weinberger}, H.~F.~Weinberger independently obtained  a particular case of Theorem~\ref{tA} which can be derived from (\ref{Sz1}) by putting $f(x)=x^p$, $p\ge 1$. Finally, H.~D.~Brunk and I.~Olkin using different technics proved a weighted version of (\ref{Sz1}).

\begin{thL}[Brunk (1956) \cite{Brunk}, Olkin (1959) \cite{Olkin}] Let
$$
0\le w_n\le w_{n-1}\le \ldots\le w_1\le 1, \qquad 0\le b_{n}\le b_{n-1}\le \cdots \le b_2 \le b_1,
$$
and a function $f=f(x)$ be convex on $[0;b_1]$, then
$$
f\left(\sum_{k=1}^{n}(-1)^{k+1}w_k b_k\right)\le \left(1-\sum_{k=1}^{n}(-1)^{k+1}w_k\right)f(0)+\sum_{k=1}^{n}(-1)^{k+1}w_kf(b_k).
$$
\label{tD}
\end{thL}

Note that Theorems \ref{tA} and \ref{tD} are of Jensen type. It was shown in \cite{Bellman1959} that these results were particular cases of more general statements. Namely, Theorem~\ref{tA} follows from the majorization theorem \cite[Theorem~108]{HLP}, and Theorem \ref{tD} is a corollary of Steffensen's inequality (see, e.g., \cite{Steffensen}). Analogues of other classical inequalities for sums with alternating signs were considered by M.~Biernacki in \cite{Biernacki}. He showed e.g. that Chebyshev's sum inequality remains valid for such sums.
Later on, some refinements of Theorems~\ref{tA} and \ref{tD} were obtained in \cite{Bjelica}, \cite{Pecaric3}, and \cite{Pecaric2}.

Note that inequalities with alternating signs have numerous applications. For instance, G. Szeg\H{o} proved Theorem~\ref{tA} for purposes in generalization of Dirichlet integrals, the result of H.~F.~Weinberger was motivated by certain problems in symmetrization theory, etc.
Close connection between sums with alternating signs and estimates of trigonometrical integrals was observed by J. F. Steffensen \cite{Steffensen}.

\medskip

Now we give some notations and auxiliaries. In what follows, we denote non-negative sequences of real numbers in bold print, e.g. $\mathbf{a}=\{a_k\}_{k=1}^{n}$ or $\mathbf{b}=\{b_k\}_{k=1}^{n}$, where $n$ is a positive integer or infinity (usually we omit number of elements). Sometimes properties of the sequences can be specified.
Expressions like $\mathbf{a}\equiv 1$ mean that all elements of $\mathbf{a}$ are equal to $1$.
From now on, we exclude cases of sequences such that denominators in inequalities for them vanish.

Further, let us recall some well-known inequalities for $\alpha,\beta\ge0$, which we use:
\begin{eqnarray}
\label{Jensen}    &(\alpha+\beta)^p\le \; 2^{p-1} (\alpha^p+\beta^p),\quad p\ge 1 &(\text{Jensen's inequality}); \\
\label{JensenRev} &(\alpha+\beta)^p\ge \; 2^{p-1} (\alpha^p+\beta^p),\quad 0<p<1  &(\text{reverse Jensen's inequality}); \\
\label{Young}     &\alpha \beta\le \frac{\alpha^p}{p}+\frac{\beta^q}{q},\quad  \frac{1}{p}+\frac{1}{q}=1,\quad p\ge 1&(\text{Young's inequality});\\
\label{ineq+-} &p\,\beta^{p-1}(\alpha-\beta)\le \alpha^p-\beta^p\le p\,\alpha^{p-1}(\alpha-\beta),&p\ge 1,\quad \alpha\ge \beta;\\
\label{ineqS} &\alpha^p+\beta^p\le (\alpha+\beta)^p,\quad & p\ge 1.
 \end{eqnarray}
The inequality~(\ref{ineq+-}) can be obtained through dividing both parts by $\alpha-\beta$; we also refer the reader to \cite[Theorem 41]{HLP}. The following result will be needed in Section~\ref{Sec}.
\begin{lemma}
\label{lemmaM}
Let $\mathbf{a}$ be non-increasing, $\mathbf{b}$ be non-decreasing and such that $b_k\le B$ for $k=1,\ldots,n$. Then
$$
\sum_{k=1}^n (-1)^{k+1}a_k b_k\le B \sum_{k=1}^n (-1)^{k+1}a_k.
$$
\end{lemma}
\begin{proof} Since the sequences $\mathbf{a}$ and $\{B-b_k\}$ are non-increasing, we have
$$
B \sum_{k=1}^n (-1)^{k+1}a_k-\sum_{k=1}^n (-1)^{k+1}a_k b_k=
\sum_{k=1}^n (-1)^{k+1}a_k(B-b_k)\ge 0.
$$
It is easily seen that equality holds e.g. if $\mathbf{b}\equiv B$.
\end{proof}

\section{H\"{o}lder type inequalities}
\label{Sec}
In this section, we show that there is no a direct analog of  H\"{o}lder's inequality in case with alternating signs, but
it is possible to obtain one of reverse H\"{o}lder's inequality. Note that reverse H\"{o}lder's inequalities for non-negative terms are well studied (see e.g. \cite{Zhuang} and references there in).

\begin{theorem} Let $\mathbf{a}$ and $\mathbf{b}$ be positive non-increasing and such that
$$
0<a\le a_k\le A <\infty, \qquad 0<b\le b_k\le B <\infty,\qquad k=1,\ldots,n,
$$
then for $p,q>1$, $\frac{1}{p}+\frac{1}{q}=1$, we have
\begin{equation}
0\le \frac{\left(\sum_{k=1}^n (-1)^{k+1}a_k^q\right)^{1/q}\left(\sum_{k=1}^n (-1)^{k+1}b_k^p\right)^{1/p}}{\sum_{k=1}^n (-1)^{k+1}a_k b_k}
\le C_{\mathbf{a},\mathbf{b}},
\label{BC}
\end{equation}
where $C_{\mathbf{a},\mathbf{b}}=A^{q-1}/b+B^{p-1}/a$ and $C_{\mathbf{a},\mathbf{b}}\in (1;\infty)$. The left hand side of~$(\ref{BC})$ should be read as for all even $n\ge 2$ there exists no positive constant depending on $a,A,b,B,p$ or $q$, and bounding the fraction in $(\ref{BC})$ from below.
\label{Holder1}
\end{theorem}
\begin{proof} From now on, $F_{\text{H}}$ stands for the fraction in (\ref{BC}). The fact that there exist no positive constants bounding $F_{\text{H}}$ from below, can be shown by the following example. Let $n$ be even and $\mathbf{a}=\{a_1,a_1,a_3,a_3,\ldots,a_n,a_n,\ldots\}$ be positive and non-decreasing. The sequence $\mathbf{b}$ is arbitrary except such that ${b_{2k-1}-b_{2k}=0}$ for all $k=1,\ldots,n/2$. It follows that
$$
F_{\text{H}}
=\frac{0\cdot\left(\sum_{k=1}^n (-1)^{k+1}b_k^p\right)^{1/p}}{\sum_{k=1}^{n/2} a_{2k-1}( b_{2k-1}-b_{2k})}=0.
$$
Thus $F_{\text{H}}$ cannot be bounded from below by a positive absolute constant or a constant depending on $p$, $q$, maximum or minimum elements of $\mathbf{a}$ and $\mathbf{b}$.

Now we prove the right hand side of (\ref{BC}). Here $N_{\text{H}}$ denotes the numerator of~$F_{\text{H}}$. From (\ref{Young}) we have
$$
N_{\text{H}} \le \frac{1}{q}\sum_{k=1}^n (-1)^{k+1}a_k^q+\frac{1}{p}\sum_{k=1}^n (-1)^{k+1}b_k^p.
$$
Note that any sum with alternating signs can be written in the form
\begin{equation}
\label{+-}
\sum_{k=1}^n(-1)^{k+1}\alpha_k=\sum_{k=1}^{N}(\alpha_{2k-1}-\alpha_{2k}),
\end{equation}
where $N=n/2$ if $n$ is even and $N=(n+1)/2$  if $n$ is odd (we also assume that $\alpha_{n+1}=0$). Therefore from the right hand side of (\ref{ineq+-}) we obtain
\begin{equation*}
\begin{split}
N_{\text{H}}&\le\frac{1}{q}\sum_{k=1}^N (a_{2k-1}^q-a_{2k}^q)+\frac{1}{p}\sum_{k=1}^N (b_{2k-1}^{p}-b_{2k}^p)\\
   &\le \frac{1}{q}\sum_{k=1}^N q \,a_{2k-1}^{q-1}(a_{2k-1}-a_{2k})+\frac{1}{p}\sum_{k=1}^N p \,b_{2k-1}^{p-1}(b_{2k-1}-b_{2k})\\
   &\le A^{q-1}\sum_{k=1}^n (-1)^{k+1}a_k+B^{p-1}\sum_{k=1}^n (-1)^{k+1}b_k\\
   &=\sum_{k=1}^n (-1)^{k+1}\left(\frac{A^{q-1}}{b_k}+\frac{B^{p-1}}{a_k}\right)a_k b_k.
\end{split}
\end{equation*}
In the last expression, sequences $\{A^{q-1}/b_k +B^{p-1}/a_k\}$ and $\{a_k b_k\}$ are non-decreasing and non-increasing, correspondingly,  since $\mathbf{a}$ and $\mathbf{b}$ are non-increasing by the data. Hence by Lemma
$$
N_{\text{H}}\le \max_k\left\{\frac{A^{q-1}}{b_k}+\frac{B^{p-1}}{a_k}\right\}\sum_{k=1}^n (-1)^{k+1}a_k b_k\le
C_{\mathbf{a},\mathbf{b}}\sum_{k=1}^n (-1)^{k+1}a_k b_k.
$$

We conclude the proof by observing that $C_{\mathbf{a},\mathbf{b}}\in (1,\infty)$. Indeed, it is easily seen that $C_{\mathbf{a},\mathbf{b}}\to \infty$ as $a\to 0$ or $b\to 0$. We have $C_{\mathbf{a},\mathbf{b}}>1$, because,  on the one hand, from (\ref{Young}) taking into account that $a\le A$, $b\le B$ we have
$$
C_{\mathbf{a},\mathbf{b}}\ge \frac{a^{q-1}}{b}+\frac{b^{p-1}}{a}> \frac{a^q/q+b^p/p}{ab}\ge 1,
$$
and, on the other hand, $C_{\mathbf{a},\mathbf{b}}\to 1$ from above as $b\to \infty$ if $\mathbf{a}\equiv 1$ and $p$ is sufficiently close to~$1$.
\end{proof}
\begin{remark}
From Theorem~\ref{Holder1}, it is seen that $C_{\mathbf{a},\mathbf{b}}$ tends to infinity as $a\to 0$ or $b\to 0$.
Now we give an example of sequences that confirms it.  Following the notation of Theorem~\ref{Holder1}, we suppose that the number of terms in the sums is odd, $\mathbf{a}\equiv 1$ and $b=b_{2n+1}=0$ in $\mathbf{b}$. It gives
\begin{equation*}
\begin{split}
F_{\text{H}}&=\frac{\left(\sum_{k=1}^{2n+1} (-1)^{k+1}a_k^q\right)^{1/q}\left(\sum_{k=1}^{2n+1} (-1)^{k+1}b_k^p\right)^{1/p}}
{\sum_{k=1}^{2n+1} (-1)^{k+1}a_k b_k}\\
            &= \frac{\left(\sum_{k=1}^{2n} (-1)^{k+1}b_k^p\right)^{1/p}}
{\sum_{k=1}^{2n} (-1)^{k+1} b_k}.
\end{split}
\end{equation*}
From (\ref{+-}) and the left hand side of (\ref{ineq+-}) we deduce that
$$
F_{\text{H}}=\frac{\left(\sum_{k=1}^{n} (b_{2k-1}^p-b_{2k}^p)\right)^{1/p}}
{\sum_{k=1}^{n} (b_{2k-1}-b_{2k})}\ge
p^{1/p}\left(\frac{b_{2n}}{\sum_{k=1}^{n}(b_{2k-1}-b_{2k})}\right)^{1-1/p},
$$
where the power is positive since $p>1$ and, therefore, for a fixed positive $b_{2n}$  the sum in the denominator can be made sufficiently small due to an appropriate choice of the sequence $\mathbf{b}$. Consequently, $F_{\text{H}}$ can be arbitrarily large.

On the other hand, in some particular cases $F_{\text{H}}$ can be bounded from above by an absolute constant even if $\mathbf{a}$ and $\mathbf{b}$ are decreasing to zero as $n\to\infty$. For instance, for harmonic series we have:
\begin{equation}
\label{harm1}
\left(\sum_{k=1}^\infty \frac{(-1)^{k+1}}{k^{q\alpha}}\right)^{1/q} \left(\sum_{k=1}^\infty \frac{(-1)^{k+1}}{k^{p\beta}}\right)^{1/p}\le \sum_{k=1}^\infty \frac{(-1)^{k+1}}{k^{\alpha+\beta}}, \qquad \alpha>0,\quad \beta>0.
\end{equation}
Indeed, since $\sum_{k=1}^\infty (-1)^{k+1}k^{-s}=(1-2^{1-s})\zeta(s)$ for $s>0$, in order to prove (\ref{harm1}) it is sufficient to note that the function
$$
F(\alpha,\beta)=\frac{\left((1-2^{1-q\alpha})\zeta(q\alpha)\right)^{1/q}\left((1-2^{1-p\beta})\zeta(p\beta)\right)^{1/p}}{(1-2^{1-(\alpha+\beta)})\zeta(\alpha+\beta)}
$$
has a maximum at $q\alpha=p\beta$ and moreover $\max F(\alpha,\beta)=1$. The following inequality for geometric series holds:
\begin{equation}
\label{geom1}
\left(\sum_{k=1}^\infty \frac{(-1)^{k+1}}{a^{qk}}\right)^{1/q} \left(\sum_{k=1}^\infty \frac{(-1)^{k+1}}{b^{pk}}\right)^{1/p}\le \sum_{k=1}^\infty \frac{(-1)^{k+1}}{(ab)^{k}}, \qquad a>1, \quad b>1.
\end{equation}
Certainly, the left and right hand sides of the inequality equal ${(1+a^q)^{-\frac{1}{q}}(1+b^p)^{-\frac{1}{p}}}$ and $\left(1+ab\right)^{-1}$ respectively. Consequently, we have $1+ab\le (1+a^q)^{\frac{1}{q}}(1+b^p)^{\frac{1}{p}}$, which is true by H\"{o}lder's inequality.
\end{remark}

\medskip

It is clear that if $p=q=2$ then the constant $C_{\mathbf{a},\mathbf{b}}$ from Theorem~\ref{Holder1} is equal to $A/b+B/a\ge a/b+b/a \ge 2$. Now we obtain a more precise constant for the case when the sequences $\mathbf{a}$ and $\mathbf{b}$ satisfy some additional conditions.

\begin{propos}
\label{Cauchy1}
Under the assumptions on the sequences $\mathbf{a}$ and $\mathbf{b}$ of Theorem~\ref{Holder1}, if moreover the sequence $\{a_k/b_k\}$ is monotone,
\begin{equation}
\label{CBS2}
0\le \frac{\sum_{k=1}^n(-1)^{k+1}a_k^2 \sum_{k=1}^n(-1)^{k+1}b_k^2}{\left(\sum_{k=1}^n(-1)^{k+1}a_kb_k\right)^2}\le c^2_{\mathbf{a},\mathbf{b}},
\end{equation}
where $c_{\mathbf{a},\mathbf{b}}=\frac{1}{2}\max\{A/a+a/A; B/b+b/B\}$ and $c_{\mathbf{a},\mathbf{b}}\in [1;\infty)$. The left hand side of~$(\ref{CBS2})$ should be read as for all even $n\ge 2$ there exists no positive constant depending on $a,A,b$ or~$B$, and bounding the fraction in $(\ref{CBS2})$ from below.
\end{propos}
\begin{proof} The left hand side inequality follows by the same method as in the proof of Theorem~\ref{Holder1}. Note that one can also find related examples on non-existence of a~direct analogue of Cauchy's inequality for sums with alternating signs in~\cite{Biernacki}.

Now we prove the right hand side of (\ref{CBS2}).  Since (\ref{CBS2}) is linear homogeneous, we may consider the sequences $\mathbf{a'}=\mathbf{a}/A$ and $\mathbf{b'}=\mathbf{b}/B$ such that
$$
0<a/A\le a_k'\le 1, \qquad 0<b/B\le b_k'\le 1,\qquad k=1,\ldots,n,
$$
instead of $\mathbf{a}$ and $\mathbf{b}$. Let $N_{\text{C}}$ denote the numerator of the fraction in (\ref{CBS2}) for $\mathbf{a'}$ and~$\mathbf{b'}$. Applying (\ref{Young}) with $p=q=2$ yields
\begin{equation*}
\begin{split}
N_{\text{C}}&\le \frac{1}{4}\left(\sum_{k=1}^n(-1)^{k+1}({a'_k}^2+{b'_k}^2)\right)^2\\
            &= \frac{1}{4}\left(\sum_{k=1}^n(-1)^{k+1}\left(\frac{{a'_k}}{b'_k}+\frac{b'_k}{{a'_k}}\right){a'_k} b'_k\right)^2.
\end{split}
\end{equation*}
In the last expression, the sequence $\{c_k+1/c_k\}$, where $c_k=a'_k/b'_k$, is non-decreasing. Indeed, $\{a_k/b_k\}$ is monotone (non-increasing or non-decreasing), consequently, $\{c_k\}$ is also monotone and  moreover $c_1=1$. Since $f(x)=x+1/x$ is convex for $x\in(0;\infty)$ and has a minimum at $x=1$, the sequence $\{f(c_k)\}=\{c_k+1/c_k\}$ is non-decreasing. From this by Lemma, we thus obtain
\begin{equation*}
\begin{split}
N_{\text{C}}\le \frac{1}{4}\left(\max\left\{f(c_k)\right\}\right)^2
\left(\sum_{k=1}^n(-1)^{k+1}a'_k b'_k\right)^2.
\end{split}
\end{equation*}
By convexity, $\max_{x\in[x_1,x_2]} f(x)=\max\{f(x_1),f(x_2)\}$ for each segment $x\in[x_1,x_2]$, $0<x_1\le x_2 <\infty$. Hence,
on account of the properties of $\mathbf{a}'$ and $\mathbf{b}'$, we have $a/A\le c_k \le B/b$ and consequently
$$
\max\left\{f(c_k)\right\} = \max\left\{a/A+A/a; B/b+b/B\right\}.
$$
This implies the right hand side of (\ref{CBS2}). It is easily seen that equality holds if both sequences are constant. The fact that $c_{\mathbf{a},\mathbf{b}}\ge 1$ is obvious.
\end{proof}

The inequality (\ref{CBS2}) is an analogue of the following inequality for non-negative terms, where the left hand side is just Cauchy's inequality and the right hand side is a particular case of a general result due to Y.~D.~Zhuang.
\begin{thL}[Cauchy (1821); Zhuang (1991) \cite{Zhuang}]
Let  $\mathbf{a}$ and $\mathbf{b}$ be positive and such that
$0<a\le a_k\le A <\infty$ and $0<b\le b_k\le B <\infty$ for $k=1,\ldots,n$, then
$$
1\le \frac{\sum_{k=1}^n a_k^2 \sum_{k=1}^n b_k^2}{\left(\sum_{k=1}^n a_k b_k\right)^2} \le \varsigma_{\mathbf{a},\mathbf{b}}^2 ,
$$
where $\varsigma_{\mathbf{a},\mathbf{b}}=\frac{1}{2}\max\left(A/b+b/A;a/B+B/a\right)$ and $\varsigma_{\mathbf{a},\mathbf{b}}\in[1;\infty)$.
\label{tE}
\end{thL}

\begin{remark} Assume that $\mathbf{a}\equiv 1$,  $\mathbf{b}$ is such that $0<b\le b_k\le B <\infty$ for all $k$, and $n$ is odd in Theorem~\ref{Holder1}. Then it is clear that under these assumptions Theorems~\ref{tA} and~\ref{Holder1} can be summarized as follows
$$
1 \le \frac{\left(\sum_{k=1}^{2n+1}(-1)^{k+1}b_k^p\right)^{1/p}}{\sum_{k=1}^{2n+1} (-1)^{k+1}b_k}
\le \left(b^{-1}+B^{p-1}\right), \quad p\ge 1.
$$
Let us write this inequality in another form. Suppose $b_k=x_k^r$, $p=R/r$, where $R$ and $r$ are positive integer numbers and $R\ge r$. Raising all expressions to the power $1/r$ and taking into account that $x^{-r}+X^{R-r}\le x^{-r}(1+X^R)$ implies
\begin{equation}
\label{geom}
1 \le \frac{\left(\sum_{k=1}^{2n+1}(-1)^{k+1}x_k^R\right)^{1/R}}{\left(\sum_{k=1}^{2n+1}(-1)^{k+1}x_k^r\right)^{1/r}}\le \frac{1}{x}\left(1+X^{R}\right)^{1/r}, \quad r,R\in \mathbb{N}, \quad r\le R,
\end{equation}
where $0<x\le x_k\le X <\infty$ for each $k=1,\ldots,{2n+1}$. Note that the right hand side of (\ref{geom}) cannot be appreciably improved from the point of view of Remark~1.

Following \cite{Weinberger}, we explain a geometrical meaning of (\ref{geom}), which has some applications in theory of symmetrization.
Let $x_k$ be the radii of concentric spheres in a space of dimension $R$. Then the value in the numerator is the radius of a single sphere having the total volume contained between the spheres of radius $x_1$ and $x_2$, $x_3$ and $x_4$, etc., and the value in the denominator is the equivalent radius in the same sense in a space of dimension $r$. Hence (\ref{geom}) states that the fraction of radii of these spheres cannot be small since it is bounded by a constant and can be as large as it is allowed by boundaries for the radii and dimensions $r$ and $R$.
\end{remark}

\section{Minkowski type inequalities}

In this section, we obtain sharp Minkowski type inequalities with alternating signs and sharp reverse Minkowski's inequality for non-negative terms.
\begin{theorem}
Let $\mathbf{a}$ and $\mathbf{b}$ be non-negative non-increasing, then
\begin{equation}
\label{Minn}
0\le \frac{\left(\sum_{k=1}^n (-1)^{k+1}a_k^p\right)^{1/p}+\left(\sum_{k=1}^n (-1)^{k+1} b_k^p\right)^{1/p}}{\left(\sum_{k=1}^n (-1)^{k+1}(a_k +b_k)^p\right)^{1/p}} \le 2^{1-1/p}, \quad p\ge 1.
\end{equation}
The constant $2^{1-1/p}$ is best possible. The left hand side of $(\ref{Minn})$ should be read as for all $n\ge 2$ there exists no constant depending on $p$ only and bounding the fraction in~$(\ref{Minn})$ from below.
The fraction becomes reciprocal if $0<p<1$.
\label{Mink0}
\end{theorem}
\begin{proof}
Throughout the proof,  $F_{\text{M}}$ denotes the fraction in (\ref{Minn}).
We begin by proving the left hand side of (\ref{Minn}), i.e. by proving of non-existence of positive constant depending on $p$ only, which bounds $F_{\text{M}}$ from below. To prove this, it is sufficient to make the following observation. For each $p>1$ there exists a sequence such that $F_{\text{M}}$  tends to zero. Indeed, suppose that $n\ge 2$, $\mathbf{a}=\{1,1,0,\ldots,0,\ldots\}$ and $\mathbf{b}=\{b,0,\ldots,0,\ldots\}$ with some $b>0$. From the left hand side of (\ref{ineq+-}), we deduce that
$$
F_{\text{M}}=\frac{b}{\left((1+b)^p-1\right)^{1/p}}\le\frac{b}{(pb)^{1/p}}<b^{1-\frac{1}{p}}.
$$
In this way $F_{\text{M}}\to 0$ as $b\to 0$ since $1-1/p>0$ for all $p>1$.

Let us prove the right hand side of (\ref{Minn}). From (\ref{Jensen}) we have
$$
\left(\left(\sum_{k=1}^n (-1)^{k+1}a_k^p\right)^{1/p}+\left(\sum_{k=1}^n (-1)^{k+1} b_k^p\right)^{1/p}\right)^p
\le 2^{p-1}\left(\sum_{k=1}^n (-1)^{k+1}(a_k^p+b_k^p)\right).
$$
Now it is enough to show that
\begin{equation}
\label{ineq22}
\sum_{k=1}^n (-1)^{k+1}(a_k^p+b_k^p)\le \sum_{k=1}^n (-1)^{k+1}(a_k +b_k)^p, \qquad p\ge 1,
\end{equation}
and extract the $p$\;th root. The inequality~(\ref{ineq22}) is equivalent to
$$
\sum_{k=1}^n (-1)^{k+1}\left((a_k +b_k)^p-(a_k^p+b_k^p)\right)\ge 0,
$$
which holds since $(a_k +b_k)^p-(a_k^p+b_k^p)\ge (a_{k+1} +b_{k+1})^p-(a_{k+1}^p+b_{k+1}^p)$
for each $k=1,\ldots,n-1$. The latter follows from the implication
 \begin{equation}
\label{ineq33}
f(a_k,y)\ge f(a_{k+1},y),\; f(x,b_{k})\ge f(x,b_{k+1}) \; \Rightarrow\; f(a_k,b_k)\ge f(a_{k+1},b_{k+1}),
\end{equation}
which is true since the function $f(x,y)=(x+y)^p-(x^p+y^p)$ with $p\ge 1$ is non-decreasing for $x\ge0$ and $y\ge 0$ as a function of $x$ or of $y$ with a fixed $y$ or $x$ respectively, since $f'_x\ge0$ and $f'_y\ge 0$. This completes the proof of (\ref{ineq22}).

The accuracy of the constant $2^{1-1/p}$ is verified by the following example. Let $\mathbf{a}=\{1,1,1,0,\ldots,0,\ldots\}$, $\mathbf{b}=\{b,\sqrt[p]{b^p-1},0,\ldots,0,\ldots\}$, $b>1$, and $n\ge 3$. Then
$$
F_{\text{M}}=\frac{2}{\left((1 +b)^p-(1+\sqrt[p]{b^p-1})^p+1\right)^{1/p}}=2^{1-1/p}-\varepsilon_b,
$$
where $\varepsilon_b$ is positive and $\lim_{b\to \infty}\varepsilon_b=0$ since $(1 +b)^p-(1+\sqrt[p]{b^p-1})^p\to 1$ from above as $b\to \infty$.

For $0<p<1$ the fraction $F_{\text{M}}$ should be replaced by $1/F_{\text{M}}$. Indeed, we then use (\ref{JensenRev}) instead of (\ref{Jensen}) and reversed version of  (\ref{ineq22}) since $f(x,y)=(x^p+y^p)-(x+y)^p$, $0<p<1$, is non-decreasing for $x\ge0$ and $y\ge 0$ as a function of $x$ or of $y$, hence the implication (\ref{ineq33}) is still valid for this function. The same examples are used to prove sharpness.
\end{proof}

\begin{remark}
 We leave it to the reader to verify that
(\ref{Minn}) has a weighed version if  put $a_k=w_k^{1/p} \alpha_k$ and $b_k=w_k^{1/p} \beta_k$, where
 sequences $\{\alpha_k\}$, $\{\beta_k\}$ and weights $\{w_k\}$ are non-increasing and moreover $w_1\le 1$.
\end{remark}
\begin{remark}
\label{rem}
By Theorem~\ref{Mink0}, for $0<p<1$ we have
$$
\left(\sum_{k=1}^n (-1)^{k+1}(a_k +b_k)^p\right)^{\frac{1}{p}}\le 2^{\frac{1}{p}-1}\left(\left(\sum_{k=1}^n (-1)^{k+1}a_k^p\right)^{\frac{1}{p}}+\left(\sum_{k=1}^n (-1)^{k+1} b_k^p\right)^{\frac{1}{p}}\right),
$$
where the constant is sharp. Because of this inequality, it is reasonable to suppose that a real non-negative functional
$$
\|\mathbf{x}\|=\left(\sum_{k=1}^n (-1)^{k+1}|x_k|^p\right)^{1/p}, \qquad 0<p<1,
$$
is \textit{a quasi-norm} on an appropriate (vector) space $E$ of decreasing non-negative sequences $\mathbf{x}=\{x_k\}_{k=1}^n$ (including $\mathbf{x}=0$) since it satisfies the axioms:
\begin{enumerate}
  \item $\|\mathbf{x}\|=0$ if and only if $\mathbf{x}=0$;
  \item $\|\alpha \mathbf{x}\|=|\alpha|\|\mathbf{x}\|$ for all $\alpha\in \mathbb{R}$ and $\mathbf{x}\in E$;
  \item there exist $K>0$ 
   such that $\|\mathbf{x}+\mathbf{y}\|\le K\left(\|\mathbf{x}\|+\|\mathbf{y}\|\right)$ for all $\mathbf{x},\mathbf{y}\in E$ (in~our case $K=2^{1/p-1}\ge 1$).
\end{enumerate}
However, the space $E$ equipped with regular operations of vector addition and scalar multiplication
is not a vector space (the axiom on existence of an additive inverse element is not satisfied) but just a cone for which quasi-norms are usually defined in a different way and with $K=1$ (see, e.g.,
\cite{Valero}).
\end{remark}

\medskip

Using the same technics as in the proof of Theorem~\ref{Mink0}, we can easily obtain reverse Minkowski's inequality with a sharp constant being independent of the sequences.
\begin{propos} Let $\mathbf{a}$ and $\mathbf{b}$ be non-negative and belong to $l_p$, $p\ge 1$, then
\begin{equation}
\label{myMin}
1\le \frac{\left(\sum_{k=1}^\infty a_k^p\right)^{1/p}+\left(\sum_{k=1}^\infty b_k^p\right)^{1/p}}{\left(\sum_{k=1}^\infty (a_k +b_k)^p\right)^{1/p}}\le
2^{1-1/p}, \qquad p\ge 1.
\end{equation}
Both constants are best possible.
\label{myMinkProp}
\end{propos}
\begin{proof}
The left hand side of (\ref{myMin}) is just Minkowski's inequality. Now we prove the right hand side. From (\ref{Jensen}) it follows that
$$
\left(\left(\sum_{k=1}^\infty a_k^p\right)^{1/p}+\left(\sum_{k=1}^\infty b_k^p\right)^{1/p}\right)^p\le 2^{p-1}\left(\sum_{k=1}^\infty (a_k^p+b_k^p)\right).
$$
Applying the inequality (\ref{ineqS}) and extracting the $p$\,th root of both sides give (\ref{myMin}).

The constant $2^{1-1/p}$ cannot be replaced by smaller one. Indeed, if $n$ is positive integer, then for
$\mathbf{a}=\{1,\ldots,1,0,0,\ldots\}$ (first $n$ elements are units) and $\mathbf{b}=\{n^{1/p},0,0,\ldots\}$ we obtain after some simplifications
$$
\frac{\left(\sum_{k=1}^n a_k^p\right)^{1/p}+\left(\sum_{k=1}^n b_k^p\right)^{1/p}}{\left(\sum_{k=1}^n (a_k +b_k)^p\right)^{1/p}}=
2\left(1-\frac{1}{n}+\left(1+\frac{1}{n^{1/p}}\right)^p\right)^{-1/p}=2^{1-1/p}-\varepsilon_n,
$$
where positive $\varepsilon_n\to 0$ as $n\to \infty$.
\end{proof}
\begin{remark} Analogously we can obtain integral version of $(\ref{myMin})$:
\begin{equation}
\label{Mink_int}
1\le \frac{\|f\|_p+\|g\|_p}{\|f+g\|_p}\le 2^{1-1/p}, \qquad p\ge 1,
\end{equation}
where $f,g$ belong to usual $L_{p}$-space with the norm $\|f\|_p=\left(\int f^p\right)^{1/p}$, $p\ge 1$.
\end{remark}

To the best of our knowledge, the right hand side of (\ref{myMin}) (and (\ref{Mink_int})) is not contained in common books on inequalities.
However it can be also obtained from the following  result of theory of quasi-normed spaces:
\begin{equation}
\label{quasiMin}
1\le \frac{\|g+h\|_p}{\|g\|_p+\|h\|_p}\le 2^{1/p-1}, \qquad g,h\in L_p,\qquad 0<p<1,
\end{equation}
where $f,g$ belong to quasi-Banach space $L_p$ with the quasi-norm $\|f\|_p=\left(\int f^p\right)^{1/p}$, $0<p<1$ (see e.g. \cite[Appendix H]{Benyamini}, sf. Remark~\ref{rem}). Indeed, it is sufficient to apply quasi-linearization technics as in \cite[\S 22]{BB} to (\ref{quasiMin}).

\bigskip

Let us now compare the constant $2^{1-1/p}$ from (\ref{myMin}) with ones in \cite{Bougoffa} and \cite{Castillo}, where reverse integral Minkowski's inequalities for positive functions were obtained in terms of boundaries of their quotient. At first we formulate a result from \cite{Bougoffa}.

\begin{thL}[Bougoffa \cite{Bougoffa} (a discrete version)] Let $\mathbf{a}$ and $\mathbf{b}$ be positive and such that
${0< m\le a_k/b_k \le M<\infty}$, $k=1,\ldots,n$, then
$$
\left(\sum_{k=1}^n a_k^p\right)^{1/p}+\left(\sum_{k=1}^n b_k^p\right)^{1/p}\le
C_{m,M}\left(\sum_{k=1}^n (a_k +b_k)^p\right)^{1/p}, \qquad p\ge 1,
$$
where $C_{m,M}=1+\frac{1}{m+1}-\frac{1}{M+1}$ and $C_{m,M}\in[1;2)$.
\end{thL}

It is easily seen that the right hand side of (\ref{myMin}) is more precise for all $\mathbf{a}$ and $\mathbf{b}$ as above if
$$
p <\ln 2\;/(\ln 2- \ln C_{m,M}).
$$

 For example, let $m$ be close to zero and $M$ be large enough, then $C_{m,M}=2\;(1-\varepsilon_{m,M})$, where positive $\varepsilon_{m,M}$ is small. Hence, the constant in the right hand side of (\ref{myMin}) is stronger for all $p\in[1;p^*)$, where $p^*:=- \ln 2/\ln\left(1-\varepsilon_{m,M}\right)$ can be arbitrarily large for sufficiently small $\varepsilon_{m,M}$.

\begin{remark}
Finally, we would like to mention that it would be interesting to obtain inequalities with alternating sings for another monotone type sequences such as convex and general monotone (see e.g. \cite{Tikhonov}).
\end{remark}

\bigskip

\textbf{Acknowledgements.} This research was supported by MTM 2011-27637 grant.

\bigskip

\bigskip

\bigskip

\textbf{Petr Chunaev}

Centre de Recerca Matem\`{a}tica

Campus de Bellaterra, Edifici C

08193 Bellaterra (Barcelona), Spain

Tel.: +34-93 581 1081

Fax: +34-93 581 2202

chunayev@mail.ru

\end{document}